\pdfoutput=1
\RequirePackage{ifpdf}
\ifpdf 
\documentclass[pdftex]{sigma}
\else
\documentclass{sigma}
\fi

\usepackage{bm}

\usepackage[all,2cell,cmtip]{xy}

\numberwithin{equation}{section}

\newtheorem{Theorem}{Theorem}[section]
\newtheorem{Lemma}[Theorem]{Lemma}
\newtheorem{Proposition}[Theorem]{Proposition}
\newtheorem{Conjecture}[Theorem]{Conjecture}

\theoremstyle{definition}
\newtheorem{Definition}[Theorem]{Definition}
\newtheorem{Remark}[Theorem]{Remark}
\newtheorem{Set}[Theorem]{Setup}
\newtheorem{Example}[Theorem]{Example}

\def\C{{\mathbb C}}

\def\P{{\mathbb P}}

\def\Z{{\mathbb Z}}

\def\cE{{\mathcal E}}

\def\cL{{\mathcal L}}
\def\cM{{\mathcal M}}

\def\cO{{\mathcal{O}}}

\def\cU{{\mathcal U}}

\def\operatorname#1{\mathop{\rm #1}\nolimits}

\def\codim{\operatorname{codim}}

\def\deg{\operatorname{deg}}

\def\ME{{\operatorname{ME}}}

\newcommand{\cME}[1]{\overline{\ME}}

\begin{document}

\newcommand{\arXivNumber}{2405.04002}

\renewcommand{\PaperNumber}{045}

\FirstPageHeading

\ShortArticleName{Quadratic Varieties of Small Codimension}

\ArticleName{Quadratic Varieties of Small Codimension}

\Author{Kiwamu WATANABE}

\AuthorNameForHeading{K.~Watanabe}

\Address{Department of Mathematics, Faculty of Science and Engineering, Chuo University,\\
1-13-27 Kasuga, Bunkyo-ku, Tokyo 112-8551, Japan}
\Email{\href{mailto:watanabe@math.chuo-u.ac.jp}{watanabe@math.chuo-u.ac.jp}}

\ArticleDates{Received January 28, 2025, in final form June 10, 2025; Published online June 15, 2025}

\Abstract{Let $X \subset \mathbb{P}^{n+c}$ be a nondegenerate smooth projective variety of dimension $n$ defined by quadratic equations. For such varieties, P.~Ionescu and F.~Russo proved the Hartshorne conjecture on complete intersections, which states that $X$ is a complete intersection provided that $n \geq 2c+1$. As the extremal case, they also classified $X$ with $n=2c$. In this paper, we classify $X$ with $n=2c-1$.}

\Classification{14J40; 14J45; 14M10; 14M17; 51N35}

\section{Introduction}

Let $X \subset \P^{n+c}$ be a complex nondegenerate smooth projective variety of dimension $n$. Philosophically, when the codimension $c$ is small, the structure of $X$ is subject to strong constraints. In this direction, R. Hartshorne raised his famous conjecture.
\begin{Conjecture}[\cite{Hart-ci}] If $n \geq 2c+1$, then $X$ is a complete intersection.
\end{Conjecture}

As the extremal case, a variety $X \subset \P^{n+c}$ is called a {\it Hartshorne variety} if $n=2c$ and $X$ is not a complete intersection.
The Hartshorne conjecture is still widely open. On the other hand, P.~Ionescu and F.~Russo \cite{IR} proved that the Hartshorne conjecture holds for quadratic varieties, that is, varieties scheme-theoretically defined by quadratic equations. This is a generalization of J.M.~Landsberg's result \cite[Corollary~6.29]{Lands96} obtained by local differential geometric methods. Moreover, Ionescu and Russo also proved that the only quadratic Hartshorne varieties are the $6$-dimensional Grassmann variety $G\bigl(2, \C^5\bigr) \subset \P^{9}$ and the $10$-dimensional spinor variety $S^{10} \subset \P^{15}$. Let us quickly review the brilliant ideas of Ionescu and Russo's proof of \cite{IR}. If $X\subset \P^{n+c}$ is a~quadratic variety with small codimension, then $X$ is a smooth Fano variety covered by lines; in~\cite{IR}, an essentially important tool to study $X$ is the Hilbert scheme of lines on $X$ passing through a general point $x \in X$. We denote it by $\cL_x$ and the dual vector space of the tangent space $T_xX$ by $(T_xX)^{\vee}$; then $\cL_x$ is naturally embedded into \smash{$\P\bigl((T_xX)^{\vee}\bigr)\cong \P^{n-1}$}. A notable feature of $X$ is that $\cL_x\subset \P^{n-1}$ is once again a quadratic variety, and it is scheme-theoretically defined by at most $c$ quadratic equations; moreover, if $\cL_x\subset \P^{n-1}$ is a complete intersection, then so is~${X\subset \P^N}$.

The purpose of this paper is to give a classification of quadratic varieties with $n=2c-1$.

\begin{Theorem}\label{MT}
Let $X \subset \P^{n+c}$ be a nondegenerate smooth projective quadratic variety of dimension $n$. Assume that $n=2c-1$ and $X$ is not a complete intersection. Then $X$ is projectively equivalent to one of the following:
\begin{enumerate}\itemsep=0pt
\item[$(i)$] the Segre $3$-fold $\P^1 \times \P^2 \subset \P^5$;
\item[$(ii)$] a hyperplane section of the $6$-dimensional Grassmann variety $G\bigl(2, \C^5\bigr) \subset \P^{9}$;
\item[$(iii)$] a hyperplane section of the $10$-dimensional spinor variety $S^{10} \subset \P^{15}$;
\item[$(iv)$] an $11$-dimensional Fano variety $X \subset \P^{17}$ whose Fano index is $8$. Moreover, $X$ satisfies the following:
\begin{enumerate}\itemsep=0pt
\item[$(a)$] $X$ is covered by lines.
\item[$(b)$] For a general point $x \in X$, we denote by $\cL_x$ the Hilbert scheme of lines on $X$ passing through $x \in X$. Then
\[
\tau_x\colon \ \cL_x \to \P\bigl((T_xX)^{\vee}\bigr),\qquad  [\ell] \mapsto [T_x \ell]
\]
is a closed immersion, and $\cL_x \subset \P((T_xX)^{\vee})$ is projectively equivalent to a Gushel--Mukai $6$-fold $($see Definition~$\ref{def:GM}$ for the definition of a Gushel--Mukai $6$-fold$)$.
\end{enumerate}
\end{enumerate}
\end{Theorem}

In this theorem, the assumption that $X$ is quadratic is essential because infinitely many smooth nondegenerate projective $3$-folds in $\P^5$ are not complete intersections (see, for instance, \cite[Section~2.2]{CK}). 
On the other hand, according to D.~Mumford's famous result \cite{Mum:quad}, any projective variety can be realized as a quadratic variety.
Although the same method as Ionescu and Russo \cite{IR} is used to prove Theorem~\ref{MT}, the number of subjects to be treated increases, and the discussion becomes more complicated. The results of Ionescu and Russo and the preliminary results used in the proof are summarized in Section~\ref{sec2}. Cases where the codimension of Theorem~\ref{MT} is less than or equal to $2$ are handled in Section~\ref{sec:codim2}. Cases where the codimension is greater than or equal to $3$ are treated in Section~\ref{sec:codim3}.
The author needs to determine whether the fourth case that appeared in Theorem~\ref{MT} occurs. If this variety existed, what kind of properties it would satisfy will be discussed in Section~\ref{sec5}.

\section{Preliminaries}\label{sec2}

\section*{Notation}
We employ the notation as in \cite{Har,IR,KM,RussoBk}.
\begin{itemize}\itemsep=0pt
\item We denote by $\P^n$ the projective space of dimension $n$ and by $Q^n$ a smooth quadric hypersurface of dimension $n$.
\item We denote by $G(r, \C^n)$ the Grassmann variety parametrizing $r$-dimensional linear subspaces of $\C^n$. We denote by $S^{10}$ the $10$-dimensional spinor variety, which is defined as an irreducible component of the orthogonal Grassmann variety $OG\bigl(5, \C^{10}\bigr)\subset G\bigl(5, \C^{10}\bigr)$ for a~non-degenerate quadratic form on $\C^{10}$ (see, for instance, \cite{Kuz18}).
\item For a vector space $V$, we denote by $V^{\vee}$ the dual vector space of $V$. The projectivization of $V$ is defined by
\[
\P(V):={\rm Proj}\Biggl(\bigoplus_{n=0}^{\infty}{\rm Sym}^k(V)\Biggr).
\]
\item A smooth projective variety $X$ is called {\it Fano } if the anticanonical divisor $-K_X$ is ample. For a smooth Fano variety $X$, the {\it Fano index} $i_X$ is defined as the maximal integer $r$ such that $-K_X$ is divisible by $r$ in ${\rm Pic}(X)$. The coindex of $X$ is defined by $\dim X+1-i_X$.
\item For projective varieties $X$, $Y$ and $F$, a smooth surjective morphism $f\colon X\to Y$ is called an {\it $F$-bundle} if any fiber of $f$ is isomorphic to $F$.
\item For a smooth projective variety $X$, we denote by $\rho_X$ the Picard number of $X$ and by $T_X$ the tangent bundle of $X$. 
\item For an embedded projective variety $X\subset \P^N$, we denote by ${\rm Sec}(X)\subset \P^N$ the secant variety of $X$.
\item For an embedded projective variety $X\subset \P^N$ and a point $o\in \P^{N+1}$, we denote by ${\rm Cone}(o, X)\subset \P^{N+1}$ the cone over $X$ with vertex $o$.
\item For an embedded projective variety $X\subset \P^N$ and a point $x\in X$, we denote by ${\mathbb T}_xX$ the projective tangent space of $X$ at $x$.
\end{itemize}

\begin{Set}\label{set:1}
Throughout the paper, we consider $X \subset \P^{N}$ a complex nondegenerate smooth projective variety of dimension $n$ and codimension $c$. Assume that $X \subset \P^{N}$ is scheme-theoretically defined by hypersurfaces of degrees $d_1 \geq d_2 \geq \cdots \geq d_m$. We may assume that $m$ is minimal. Let us set $d:=\sum_{i=1}^{c}(d_i-1)$.
Let $\cL$ be a family of lines on $X$. For a general point $x \in X$, we denote by $\cL_x$ the Hilbert scheme of lines on $X$ passing through $x \in X$. Denoting by $(T_xX)^{\vee}$ the dual vector space of $T_xX$, the {\it tangent map}
\[
\tau_x\colon\ \cL_x \to \P\bigl((T_xX)^{\vee}\bigr),\qquad [\ell] \mapsto [T_x \ell]
\]
is a closed immersion by \cite[Section~2.2.1]{RussoBk}.
 Throughout the paper, $\cL_x$ is viewed as a closed subscheme of $\P((T_xX)^{\vee})\cong \P^{n-1}$.
When $X$ is covered by lines, put $p:=\deg N_{\ell/X}$ for any line $[\ell] \in \cL_x$.
\end{Set}

\begin{Definition}
When $X$ is scheme-theoretically an intersection of quadric hypersurfaces, i.e., $d=c$, $X$ is called {\it quadratic}.
\end{Definition}

\begin{Example} For positive integers $r<n$, the Grassmann variety $G(r, \C^n)$ is scheme-theo\-reti\-cal\-ly an intersection of quadrics via the Pl\"ucker embedding. The quadrics are given by the Pl\"ucker relations. More generally, W.~Lichtenstein \cite{Lich82} showed that for every fundamental representation of a semisimple linear algebraic group, the orbit of the highest vector is cut out by a~system of quadrics. Moreover, Mumford~\cite{Mum:quad} proves that any projective variety can be realized as a~quadratic variety.
\end{Example}

We start to recall some classical results of projective geometry. The following is a refinement of Faltings' theorem~\cite{Fal81} due to Netsvetaev.

\begin{Theorem}[{Netsvetaev's criterion \cite[{Theorem~3.2}]{Net}}]\label{them:Net} Let $X$ be a variety as in Setup~$\ref{set:1}$ and assume that $m \leq n+1$. If \smash{$m < N-\frac{2}{3}n$} or \smash{$n\geq \frac{3}{4}N- \frac{1}{2}$}, then $X$ is a complete intersection.
\end{Theorem}

\begin{Theorem}[{Zak's theorem on linear normality \cite[Chapter~II, Corollary~2.15]{Z}}]\label{them:Zak:Lin} Let $X$ be a~variety as in Setup~$\ref{set:1}$. If \smash{$N < \frac{3}{2}n+2$}, then ${\rm Sec}(X)=\P^N$.
\end{Theorem}

\begin{Theorem}[{Zak's classification of Severi varieties \cite[Chapter~IV, Theorem~4.7]{Z}}]\label{them:Zak:Sev} Let $X$ be a variety as in Setup~$\ref{set:1}$. If \smash{$N = \frac{3}{2}n+2$} and ${\rm Sec}(X) \neq \P^N$, then $X$ is projectively equivalent to one of the following:
\begin{enumerate}\itemsep=0pt
\item[$(i)$] the Veronese surface $v_2\bigl(\P^2\bigr) \subset \P^5$;
\item[$(ii)$] the Segre $4$-fold $\P^2 \times \P^2 \subset \P^8$;
\item[$(iii)$] the Grassmann variety $G\bigl(2, \C^6\bigr) \subset \P^{14}$;
\item[$(iv)$] the $E_6(\omega_1) \subset \P^{26}$, which is the projectivization of the highest weight vector orbit in the $27$-dimensional irreducible representation of a simple algebraic group of Dynkin type $E_6$.
\end{enumerate}
\end{Theorem}

\begin{Lemma}\label{lem:irr} For a variety $X$ as in Setup~$\ref{set:1}$, the following hold:
\begin{enumerate}\itemsep=0pt
\item[$(i)$] If $X$ is covered by lines, then $p=\dim \cL_x$.
\item[$(ii)$] If $2p \geq n-1$, then $\cL_x \subset \P^{n-1}$ is smooth, irreducible and nondegenerate.
\end{enumerate}
\end{Lemma}

\begin{proof} See \cite[Proposition~1.5, Theorems~1.4 and~2.5]{Hw}.
\end{proof}

\begin{Theorem}[{\cite[Theorems~2.4 and~3.8]{IR}}]\label{them:IR} Let $X$ be a variety as in Setup~$\ref{set:1}$. Assume that~$X$ is quadratic.
\begin{enumerate}\itemsep=0pt
\item[$(i)$] If $n \geq c$, then $X$ is Fano.
\item[$(ii)$] If $n \geq c+1$, then $X$ is covered by lines. Moreover, $\cL_x \subset \P^{n-1}$ is scheme-theoretically defined by $c$ independent quadratic equations.
\item[$(iii)$] If $n \geq c+2$, then $X$ is a Fano variety with ${\rm Pic}(X) \cong \Z$ and $i_X=p+2$. Furthermore, the following are equivalent to each other:
\begin{enumerate}\itemsep=0pt
\item[$(a)$] $X \subset \P^N$ is a complete intersection.
\item[$(b)$] $\cL_x \subset \P^{n-1}$ is a complete intersection of codimension $c$.
\item[$(c)$] $p=n-1-c$.
\end{enumerate}
\end{enumerate}
\end{Theorem}

As a byproduct of Theorem~\ref{them:IR}, Ionescu and Russo proved that the Hartshorne conjecture on complete intersections holds for quadratic varieties. They also classified Hartshorne varieties among quadratic varieties.

\begin{Theorem}[{\cite[Theorems~3.8 and~3.9]{IR}}]\label{them:HC} Let $X$ be a variety as in Setup~$\ref{set:1}$. Assume that~$X$ is quadratic.
\begin{enumerate}\itemsep=0pt
\item[$(i)$] The Hartshorne conjecture for quadratic varieties: If $n \geq 2c+1$, then $X$ is a complete intersection.
\item[$(ii)$] Classification of Hartshorne varieties for quadratic varieties: If $n =2c$ and $X$ is not a~complete intersection, then $X$ is projectively equivalent to one of the following:
\begin{enumerate}\itemsep=0pt
\item[$(a)$] the $6$-dimensional Grassmann variety $G\bigl(2, \C^5\bigr) \subset \P^{9}$;
\item[$(b)$] the $10$-dimensional spinor variety $S^{10} \subset \P^{15}$.
\end{enumerate}
\end{enumerate}
\end{Theorem}

Finally, we conclude this section by reviewing the classification of smooth Fano varieties with large indices.

\begin{Theorem}[\cite{Fuj1,Fuj2}]\label{them:dP} Let $X$ be an $n$-dimensional smooth Fano variety with index $i_X=n-1$ $($i.e., a smooth del Pezzo variety$)$, whose Picard group is generated by a very ample line bundle. Then $X$ is isomorphic to one of the following:
\begin{enumerate}\itemsep=0pt
\item[$(i)$] a hypersurface of degree $3$;
\item[$(ii)$] a complete intersection of two quadric hypersurfaces;
\item[$(iii)$] a linear section of the Grassmann variety $G\bigl(2,\C^5\bigr) \subset \P\bigl(\bigl(\bigwedge^2\C^5\bigr)^{\vee}\bigr)$.
\end{enumerate}
\end{Theorem}

\begin{Theorem}[\cite{Mu}]\label{them:Mukai} Let $X$ be an $n$-dimensional smooth Fano variety with index $i_X=n-2$ $($i.e., a smooth Mukai variety$)$, whose Picard group is generated by a very ample line bundle. If~${n\geq 4}$, then $X$ is isomorphic to one of the following:
\begin{enumerate}\itemsep=0pt\setlength{\leftskip}{0.1cm}
\item[$(i)$] a hypersurface of degree $4$;
\item[$(ii)$] a complete intersection of a smooth quadric hypersurface and a cubic hypersurface;
\item[$(iii)$] a complete intersection of three quadric hypersurfaces;
\item[$(iv)$] a linear section of ${\sum}^6_{10} \subset \P^{10}$, where ${\sum}^6_{10} \subset \P^{10}$ is a smooth section of cone $\tilde{G}:=\smash{{\rm Cone}\bigl(o,G\bigl(2,\C^5\bigr)\bigr) \subset \P^{10}}$ over the Grassmann variety \smash{$G\bigl(2,\C^5\bigr) \!\subset\! \P\bigl(\bigl(\bigwedge^2 \C^5\bigr)^{\vee}\bigr)$} by a~smooth quadric hypersurface;
\item[$(v)$] a linear section of the $10$-dimensional spinor variety $S^{10}$;
\item[$(vi)$] a linear section of the Grassmann variety \smash{$G\bigl(2,\C^6\bigr) \subset \P\bigl(\bigl(\bigwedge^2 \C^6\bigr)^{\vee}\bigr)$};
\item[$(vii)$] a linear section of $LG\bigl(3,\C^6\bigr)$, where $LG\bigl(3,\C^6\bigr)$ is the Lagrangian Grassmann variety, which is the variety of isotropic $3$-planes for a non-degenerate skew-symmetric bilinear form on $\C^6$;
\item[$(viii)$] the $G_2$-variety which is the variety of isotropic $5$-planes for a non-degenerate skew-sym\-met\-ric $4$-linear form on $\C^7$.
\end{enumerate}
\end{Theorem}

\begin{Definition}\label{def:GM} In Theorem~\ref{them:Mukai}, ${\sum}^6_{10} \subset \P^{10}$ is called the {\it Gushel--Mukai $6$-fold}.
\end{Definition}
We refer the reader to \cite{KP18} for the geometry of the Gushel--Mukai $6$-fold.

\section[The case c leq 2]{The case $\boldsymbol{c \leq 2}$}\label{sec:codim2}

For $n=2c-1$ and $c \leq 2$, let $X$ be a variety as in Setup~$\ref{set:1}$. Assume that $X$ is quadratic and not a complete intersection. Since we assume that $X$ is not a complete intersection, we have~${c=2}$. Then $n=3$. The purpose of this section is to prove the following.

\begin{Proposition}
 Let $X$ be a variety as in Setup~$\ref{set:1}$. Assume that $X$ is quadratic and not a~complete intersection. If $(n, c)=(3,2)$, then $X \subset \P^5$ is projectively equivalent to the Segre $3$-fold~$\P^1 \times \P^2 \subset \P^5$.
\end{Proposition}

\begin{proof} As in the proof of \cite[Theorem~2.4]{IR}, we may find quadrics $Q_1$, $Q_2$ such that $X$ is an irreducible component of the complete intersection scheme $Q_1 \cap Q_2$. Since $X$ is quadratic and not a complete intersection, this yields that $\deg X=3= {\rm codim} X+1$, that is, $X$ is a variety of minimal degree. By Bertini's theorem \cite{Bel} (see also \cite{EH87}), $X$ is projectively equivalent to the Segre $3$-fold $\P^1 \times \P^2 \subset \P^5$.
\end{proof}

\section[The case c geq 3]{The case $\boldsymbol{c \geq 3}$}\label{sec:codim3}
\subsection{General properties}

In this section, we work in the following setting.

\begin{Set}\label{set:3} Let $X \subset \P^N$ be a variety as in Setup~$\ref{set:1}$. Assume that $X$ is quadratic and not a~complete intersection. Assume that $n=2c-1$ and $c \geq 3$.
\end{Set}

\begin{Lemma}\label{lem:1} Let $X \subset \P^N$ be a variety as in Setup~$\ref{set:3}$. Then the following hold:
\begin{enumerate}\itemsep=0pt
\item[$(i)$] $X \subset \P^N$ is a smooth Fano variety covered by lines such that ${\rm Pic}(X) \cong \Z$ and $i_X=p+2$.
\item[$(ii)$] $\cL_x \subset \P^{n-1}$ is scheme-theoretically defined by $c$-independent quadratic equations.
\end{enumerate}
\end{Lemma}

\begin{proof} Since $n= 2c-1 \geq c+2$, this follows from Theorem~\ref{them:IR}.
\end{proof}

\begin{Lemma}\label{lem:nondeg} Let $X \subset \P^N$ be a variety as in Setup~$\ref{set:3}$. Then we have \[2 \dim \cL_x \geq n-1.\]
In particular, $\cL_x \subset \P^{n-1}$ is smooth, irreducible, and nondegenerate.
\end{Lemma}

\begin{proof} Since $\cL_x \subset \P^{n-1}$ is defined by $c$ quadratic equations, we have an inequality
\[
\dim \cL_x \geq n-1-c =\frac{n-3}{2}.
\]
If the lower bound is attained, it follows from Theorem~\ref{them:IR} that $X$ is a complete intersection. This is a contradiction. So we have an inequality $\dim \cL_x >\frac{n-3}{2}$. Since $n=2c-1$, we have
\[
\dim \cL_x>\frac{(2c-1)-3}{2}=c-2.
\]
As a consequence, we obtain
\[
\dim \cL_x \geq c-1=\frac{n-1}{2}.
\]
The latter part of our assertion follows from Lemma~\ref{lem:irr}.
\end{proof}

\begin{Proposition}\label{prop:dimLx} Let $X \subset \P^N$ be a variety as in Setup~$\ref{set:3}$. Then we have
\[
\frac{3}{2}c-3 \leq \dim \cL_x < \frac{3}{2}c-2.
\]
\end{Proposition}

\begin{proof} By Theorem~\ref{them:IR}\,(iii), $\cL_x$ is not a complete intersection. As we have seen, $\cL_x \subset \P^{n-1}$ is scheme-theoretically defined by $c$-independent quadratic equations. Moreover, by Lemma~\ref{lem:nondeg}, we have $c \leq \dim \cL_x +1$. Thus, applying Theorem~\ref{them:Net} to $\cL_x \subset \P^{n-1}$, we have an inequality
\[
\frac{3n-9}{4} \leq \dim \cL_x < \frac{3n-5}{4}.
\]
Since $n=2c-1$, we obtain the desired inequality.
\end{proof}

\subsection[The case c is even]{The case $\boldsymbol{c}$ is even}

Let us additionally assume that $c$ is even.

\begin{Set}\label{set:4} Let $X \subset \P^N$ be a variety as in Setup~$\ref{set:3}$, and assume that $c$ is even.
\end{Set}

In this subsection, we shall prove the following:

\begin{Theorem}
Let $X$ be a variety as in Setup~$\ref{set:3}$. Assume that $c$ is even. Then $c$ is equal to $6$ and $\cL_x\subset \P^{10}$ is projectively equivalent to the Gushel--Mukai $6$-fold ${\sum}^6_{10} \subset \P^{10}$.
\end{Theorem}

\begin{Lemma}\label{lem:2:dim}
 Let $X \subset \P^N$ be a variety as in Setup~$\ref{set:4}$. Then the following hold:
\begin{enumerate}\itemsep=0pt
\item[$(i)$] $\dim \cL_x=\frac{3}{2}c-3.$
\item[$(ii)$] $\dim \cL_x -2\codim_{\P^{n-1}} \cL_x-1=\frac{1}{2}c-6.$
\end{enumerate}
\end{Lemma}

\begin{proof} The first part directly follows from Proposition~\ref{prop:dimLx}. The second part follows from the first.
\end{proof}

\begin{Lemma}\label{lem:2:c} Let $X \subset \P^N$ be a variety as in Setup~$\ref{set:4}$. Then $c$ is $4$, $6$, $8$ or $10$.
\end{Lemma}

\begin{proof} By Theorem~\ref{them:IR}\,(iii), $\cL_x$ is not a complete intersection. Then Theorem~\ref{them:HC} and Lem\-ma~\ref{lem:2:dim} yield $\frac{1}{2}c-6<0$. Hence, $c<12$. 
\end{proof}

By Lemma~\ref{lem:2:c}, the pair $(c, n, \dim \cL_x, i_X)$ satisfies one of the following.

\begin{center}
\begin{tabular}{|c|c|c|c|}
\hline
 $c$ & $n$ & $\dim \cL_x$ & $i_X$ \\ \hline \hline
 $4$ & $7$ & $3$ & $5$ \\
 $6$ & $11$ & $6$ & $8$ \\
 $8$ & $15$ & $9$ & $11$ \\
 $10$ & $19$ & $12$ & $14$ \\

 \hline
\end{tabular}

\end{center}

\begin{Lemma}
Let $X \subset \P^N$ be a variety as in Setup~$\ref{set:4}$. Then $c$ is not $4$.
\end{Lemma}

\begin{proof} Assume to the contrary that $c=4$. Then $X$ is a $7$-dimensional smooth Fano variety~${X\subset \P^{11}}$ of coindex $3$. According to Theorem~\ref{them:Mukai}, $X$ is isomorphic to a linear section of the $10$-dimensional spinor variety $S^{10}\subset \P^{15}$ or a linear section of the Grassmann variety $G\bigl(2, \C^6\bigr) \subset \P^{14}$. Theorem~\ref{them:Zak:Lin} shows that these varieties cannot be isomorphically projected into~$\P^{11}$.
\end{proof}

\begin{Proposition}\label{prop:VMRT} Let $X$ be a variety as in Setup~$\ref{set:4}$. Assume that $c$ is $6$, $8$ or $10$. Then $\cL_x \subset \P^{2c-2}$ satisfies the following:
\begin{enumerate}\itemsep=0pt
\item[$(i)$] $\cL_x \subset \P^{2c-2}$ is a $\bigl(\frac{3}{2}c-3\bigr)$-dimensional nondegenerate smooth Fano variety.
\item[$(ii)$] $\cL_x \subset \P^{2c-2}$ is scheme-theoretically defined by $c$ independent quadratic equations.
\item[$(iii)$] $\cL_x \subset \P^{2c-2}$ is covered by lines.
\item[$(iv)$] ${\rm Pic} (\cL_x)$ is isomorphic to $\Z$.
\end{enumerate}
Furthermore, for a general point ${[\ell]} \in \cL_x \subset \P^{2c-2}$, let $\cM_{{[\ell]}} \subset \P^{\frac{3}{2}c-4}$ be the Hilbert scheme of lines passing through ${[\ell]} \in \cL_x$. Then \smash{$\cM_{{[\ell]}} \subset \P^{\frac{3}{2}c-4}$} satisfies the following:
\begin{enumerate}\itemsep=0pt
\item[$(v)$] \smash{$\cM_{{[\ell]}} \subset \P^{\frac{3}{2}c-4}$} is scheme-theoretically defined by $\bigl(\frac{1}{2}c+1\bigr)$ independent quadratic equations.
\item[$(vi)$] $\frac{3}{2}c-5\geq i(\cL_x)=\dim \cM_{{[\ell]}}+2 \geq c-2$.
\end{enumerate}
\end{Proposition}

\begin{proof} In Lemmas \ref{lem:1}, \ref{lem:nondeg} and \ref{lem:2:dim}, we have already seen (i) and (ii) except that $\cL_x$ is Fano. Since
\[
\dim\cL_x -\codim_{\P^{2c-2}}\cL_x -2=c-6 \geq 0,
\]
Theorem~\ref{them:IR} implies (i)--(iv). Applying Theorem~\ref{them:IR} to $\cL_x \subset \P^{2c-2}$, $\cM_{{[\ell]}} \subset \P^{\frac{3}{2}c-4}$ is defined by~$\bigl(\frac{1}{2}c+1\bigr)$ independent quadratic equations.
Therefore, (v) holds. Moreover, it follows from~(v)
\[
\dim \cM_{{[\ell]}} \geq \biggl(\frac{3}{2}c-4\biggr)-\biggl(\frac{1}{2}c+1\biggr)=c-5.
\]
Since $X$ is not a complete intersection, Theorem~\ref{them:IR}\,(iii) yields that neither \smash{$\cL_x \subset \P^{2c-2}$} nor \smash{${\cM_{{[\ell]}} \subset \P^{\frac{3}{2}c-4}}$} is a complete intersection and $\dim \cM_{{[\ell]}}\geq c-4$.
By the Kobayashi--Ochiai theorem~\cite{KO},
\[
\frac{3}{2}c-2=\dim \cL_x+1 \geq i(\cL_x).
\] Furthermore if $i(\cL_x)\geq \frac{3}{2}c-3$, then $\cL_x$ is isomorphic to \smash{$\P^{\frac{3}{2}c-3}$} or \smash{$Q^{\frac{3}{2}c-3}$}. Then $\cL_x \subset \P^{2c-2}$ is a complete intersection, because $\cL_x \subset \P^{2c-2}$ is covered by lines. This is a contradiction. Hence, we have $\frac{3}{2}c-4 \geq i(\cL_x)$.

Suppose $i(\cL_x)= \frac{3}{2}c-4$. Then $\cL_x$ is a del Pezzo variety. By our assumption, $c$ equals~$6$,~$8$ or~$10$. Thus, according to Theorem~\ref{them:dP}, $\cL_x$ is isomorphic to $G\bigl(2, \C^5\bigr) \subset \P^{9}$. Since~${\cL_x \subset \P^{2c-2}}$ is covered by lines, the embedding is given by the ample generator of ${\rm Pic}(\cL_x)$. Then Theorem~\ref{them:Zak:Lin} yields that $\cL_x \subset \P^{2c-2}$ is projectively equivalent to $G\bigl(2, \C^5\bigr) \subset \P^{9}$. However, this is a~contradiction, because ${2c - 2\geq 10>9}$.
\end{proof}

\begin{Proposition}
Let $X$ be a variety as in Setup~$\ref{set:4}$. If $c$ is equal to $6$, then $\cL_x\subset \P^{10}$ is projectively equivalent to the Gushel--Mukai $6$-fold ${\sum}^6_{10} \subset \P^{10}$.
\end{Proposition}

\begin{proof} Assume that $c=6$. By Proposition~\ref{prop:VMRT}, $\cL_x \subset \P^{10}$ is a nondegenerate $6$-dimensional Fano variety with coindex $3$ and covered by lines; it follows from Theorem~\ref{them:Mukai} that $\cL_x$ is an isomorphic projection of one of the varieties as follows:
\begin{enumerate}\itemsep=0pt
\item[$({\rm i})$] the Gushel--Mukai $6$-fold ${\sum}^6_{10} \subset \P^{10}$;
\item[$({\rm ii})$] a linear section of the $10$-dimensional spinor variety $S^{10} \subset \P^{15}$;
\item[$({\rm iii})$] a linear section of the $8$-dimensional Grassmann variety $G\bigl(2, \C^6\bigr) \subset \P^{14}$;
\item[$({\rm iv})$] the $6$-dimensional Lagrangian Grassmann variety $LG\bigl(3, \C^6\bigr) \subset \P^{13}$.
\end{enumerate}
By Theorems~\ref{them:Zak:Lin} and \ref{them:Zak:Sev}, the varieties (ii)--(iv) cannot be isomorphically projected into $\P^{10}$. It~turns out that $\cL_x$ is projectively equivalent to the Gushel--Mukai $6$-fold ${\sum}^6_{10} \subset \P^{10}$. 
\end{proof}

\begin{Proposition}
Let $X$ be a variety as in Setup~$\ref{set:4}$. Then $c$ is not equal to $8$.
\end{Proposition}

\begin{proof} Assume that $c=8$. From Theorem~\ref{prop:VMRT}, it follows that $\dim \cM_{{[\ell]}}=4$ or $5$, and $\cM_{{[\ell]}} \subset \P^8$ is defined by $5$ quadratic equations. By Lemma~\ref{lem:irr}, $\cM_{{[\ell]}} \subset \P^8$ is smooth, irreducible and nondegenerate. If $\dim \cM_{{[\ell]}}=4$, then Theorem~\ref{them:Net} yields that $\cM_{{[\ell]}} \subset \P^8$ is a complete intersection. This contradicts our assumption that $X$ is not a complete intersection. Thus we have $\dim \cM_{{[\ell]}}=5$. By Proposition~\ref{prop:VMRT}, $\cL_x\subset {\P^{14}}$ satisfies the following:
\begin{enumerate}\itemsep=0pt
\item[$({\rm i})$] $\cL_x\subset \P^{14}$ is a nondegenerate smooth Fano $9$-fold of index $7$;
\item[$({\rm ii})$] $\cL_x\subset \P^{14}$ is covered by lines and ${\rm Pic} (X)\cong \Z$;
\item[$({\rm iii})$] $\cL_x\subset \P^{14}$ is scheme-theoretically defined by $8$ quadrics.
\end{enumerate}
By Theorem~\ref{them:Mukai}, $\cL_x\subset \P^{14}$ is projectively equivalent to a hyperplane section of the $10$-dimensional spinor variety $S^{10}\subset \P^{15}$. This contradicts Proposition~\ref{prop:spinor} below.
\end{proof}

\begin{Proposition}\label{prop:spinor} Let $V$ be a smooth hyperplane section of the $10$-dimensional spinor variety $S^{10}\subset \P^{15}$. Then the following hold:
\begin{enumerate}\itemsep=0pt
\item[$(i)$] $S^{10}\subset \P^{15}$ is not scheme-theoretically defined by $8$ quadrics.
\item[$(ii)$] $V \subset \P^{14}$ is not scheme-theoretically defined by $8$ quadrics.
\end{enumerate}
\end{Proposition}

\begin{proof} (i) By \cite[4.4]{ESB}, the linear system of quadrics containing $S^{10}\subset \P^{15}$ provides a rational map \smash{$\P^{15} \overset{\pi}\dashrightarrow Q^8 $}, that can be resolved via the blow-up of $\P^{15}$ along $S^{10}$, which is a $\P^{7}$-bundle $q\colon {\rm Bl}_{S^{10}}\bigl(\P^{15}\bigr)\to Q^8$ over $Q^8$ as follows:
\[
 \xymatrix{
 & {\rm Bl}_{S^{10}}(\P^{15}) \ar[ld]_p \ar[rd]^q & \\
 S^{10}\subset \P^{15} \ar@{.>}[rr]^{\pi}& & Q^8\subset \P^9.
 }
\]
Denoting the exceptional divisor of $p$ by $E$, we have
\begin{gather}\label{isom:spinor}
p^{\ast}\cO_{\P^{15}}(2)\otimes \cO_X(-E)\cong q^{\ast}\cO_{Q^8}(1).
\end{gather}
For a quadric hypersurface $D \subset \P^{15}$ containing $S^{10}$, we denote by $\widetilde{D}$ the strict transform of $D$ with respect to $p$ and by $H_D$ the closure of \smash{$\pi\bigl(D\setminus S^{10}\bigr)$}. By~(\ref{isom:spinor}), we have
\[
\widetilde{D}=p^{\ast}(D)-E=q^{\ast}(H_D).
\]
Assume that there exist $8$ quadrics $D_1, D_2, \dots, D_8 \in |I_{S^{10}/\P^{15}}(2)|$ such that $S$ is a scheme-theoretic intersection of $D_i$'s, that is, \smash{$S=\bigcap_{i=1}^8D_i$}.
This means that the base locus of the linear system $\langle D_1, D_2, \dots, D_8 \rangle$ coincides with \smash{$S^{10}$}: \[{\rm Bs}(\langle D_1, D_2, \dots, D_8 \rangle)=S^{10}.\]
Then we see that
\[
\bigcap_{i=1}^8 \widetilde{D_i}=\varnothing \qquad \text{in} \ {\rm Bl}_{S^{10}}\bigl(\P^{15}\bigr).
\]
However, we have
\[
\bigcap_{i=1}^8 H_i\neq \varnothing \qquad \text{in}\ Q^8\subset \P^{9}.
\]
This yields
\[
\bigcap_{i=1}^8 \widetilde{D_i}=\bigcap_{i=1}^8 q^{-1}(H_{i})=q^{-1}\Biggl(\bigcap_{i=1}^8 H_{i}\Biggr)\neq \varnothing.
\]
This is a contradiction.

(ii) Let $H \subset \P^{15}$ be a hyperplane such that $V=S^{10}\cap H$ as a scheme. Then we have a~standard exact sequence of ideal sheaves
\begin{gather}\label{ex:ideal}
0\to I_{S^{10}/\P^{15}}(1)\to I_{S^{10}/\P^{15}}(2) \to I_{V/H}(2)\to 0.
\end{gather}
Since the embedding $S^{10} \hookrightarrow \P^{15}$ is given by the complete linear system $\cO_{S^{10}}(1)$, we obtain $H^1\bigl(I_{S^{10}/\P^{15}}(1)\bigr)=0$. The above exact sequence (\ref{ex:ideal}) yields a surjection
\begin{gather}\label{surj:spinor}
H^0\bigl(S^{10}, I_{S^{10}/\P^{15}}(2)\bigr) \twoheadrightarrow H^0\bigl(S^{10}, I_{V/H}(2)\bigr).
\end{gather}

To prove our assertion, assume the contrary, that is, $V \subset \P^{14}$ is scheme-theoretically defined by $8$ quadrics $D_1', \dots, D_8'$. From the surjectivity of the above map (\ref{surj:spinor}), these quadrics $D_1', \dots, D_8'$ can be extended to quadrics \smash{$\tilde{D_1'}, \dots, \tilde{D_8'}$} containing $S^{10}$. This means that $S^{10}$ is contained in~the~scheme-theoretic intersection \smash{$\tilde{D_1'}\cap \dots \cap\tilde{D_8'}$}. Since $S^{10}$ does not coincide with \smash{$\tilde{D_1'}\cap \dots \cap\tilde{D_8'}$} by (i) and \smash{$\tilde{D_1'}\cap \dots \cap\tilde{D_8'}\cap H=V$}, we obtain a contradiction.
\end{proof}

\begin{Proposition}
Let $X$ be a variety as in Setup~$\ref{set:4}$. Then $c$ is not equal to $10$.
\end{Proposition}

\begin{proof} Assume that $c=10$. From Theorem~\ref{prop:VMRT}, it follows that $\dim \cM_{{[\ell]}}=6, 7$ or $8$, and $\cM_{{[\ell]}} \subset \P^{11}$ is defined by $6$ quadratic equations. By Lemma~\ref{lem:irr}, $\cM_{{[\ell]}} \subset \P^{11}$ is smooth, irreducible and nondegenerate. Then Theorem~\ref{them:Net} implies that $\cM_{{[\ell]}} \subset {\P^{11}}$ is a complete intersection. This contradicts our assumption that $X$ is not a complete intersection.
\end{proof}

\subsection[The case c is odd]{The case $\boldsymbol{c}$ is odd}

Let us consider the case where $c$ is odd.
In this subsection, we shall prove the following.

\begin{Theorem}
Let $X$ be a variety as in Setup~$\ref{set:3}$. Assume that $c$ is odd. Then $X$ is projectively equivalent to
\begin{enumerate}\itemsep=0pt
\item[$(i)$] a hyperplane section of the $10$-dimensional spinor variety $S^{10} \subset \P^{15}$ or
\item[$(ii)$] a hyperplane section of the $6$-dimensional Grassmann variety $G\bigl(2, \C^5\bigr) \subset \P^{9}$.
\end{enumerate}
\end{Theorem}

The following lemmas can be proved like {Lemmas~\ref{lem:2:dim} and \ref{lem:2:c}}. So, the proofs are left to the readers.

\begin{Lemma}
Let $X \subset \P^N$ be a variety as in Setup~$\ref{set:3}$ and assume that $c$ is odd. Then the following hold:
\begin{enumerate}\itemsep=0pt
\item[$(i)$] $\dim \cL_x=\frac{3c-5}{2}.$
\item[$(ii)$] $\dim \cL_x -2\codim_{\P^{n-1}} \cL_x-1=\frac{c-9}{2}.$
\end{enumerate}
\end{Lemma}

\begin{Lemma}\label{lem:2:c:o} Let $X \subset \P^N$ be a variety as in Setup~$\ref{set:3}$ and assume that $c$ is {odd}. Then $c$ is $3$, $5$ or~$7$.
\end{Lemma}

By Lemma~\ref{lem:2:c:o}, the pair $(c, n, \dim \cL_x, i_X)$ satisfies one of the following.
\begin{center}
\begin{tabular}{|c|c|c|c|}
\hline
 $c$ & $n$ & $\dim \cL_x$ & $i_X$ \\ \hline \hline
 $3$ & $5$ & $2$ & $4$ \\
 $5$ & $9$ & $5$ & $7$ \\
 $7$ & $13$ & $8$ & $10$ \\

 \hline
\end{tabular}

\end{center}

\begin{Proposition}
Let $X$ be a variety as in Setup~$\ref{set:3}$. Then the following hold:
\begin{enumerate}\itemsep=0pt
\item[$(i)$] If $c=3$, then $X \subset \P^8$ is projectively equivalent to a hyperplane section of the Grassmann variety $G\bigl(2, \C^5\bigr)\subset \P^{9}$.
\item[$(ii)$] If $c=5$, then $X \subset \P^{14}$ is projectively equivalent to a hyperplane section of the spinor variety $S^{10}\subset \P^{15}$.
\end{enumerate}
\end{Proposition}

\begin{proof} Suppose $c=3$. Then $X$ is a del Pezzo $5$-fold. By Theorem~\ref{them:dP}, $X$ is isomorphic to a~hyperplane section of the Grassmann variety $G\bigl(2, \C^5\bigr)\subset \P^{9}$. Since $X$ is covered by lines and~${c=3}$, our assertion holds. The case where $c=5$ also follows from Theorem~\ref{them:Mukai} and the same argument as above.
\end{proof}

\begin{Proposition}
 Let $X$ be a variety as in Setup~$\ref{set:3}$. Then $c$ is not equal to $7$.
\end{Proposition}

\begin{proof} Suppose $c=7$. Then, the same proof as in Proposition~\ref{prop:VMRT} shows that $\cL_x \subset \P^{12}$ is a~nondegenerate smooth quadratic variety covered by lines. Moreover we also see that the family of lines $\cM_{{[\ell]}}\subset \P^{7}$ passing through a general point ${[\ell]} \in \cL_x$ satisfies
\begin{itemize}\itemsep=0pt
\item $\cM_{{[\ell]}}\subset \P^{7}$ is defined by $4$ independent quadratic equations, and
\item $\dim \cM_{{[\ell]}}=3, 4 $ or $5$.
\end{itemize}
Then, according to Theorem~\ref{them:Net}, we see that $\cM_{[\ell]}\subset \P^{7}$ is a complete intersection. This contradicts our assumption that $X \subset \P^N$ is not a complete intersection.
\end{proof}

\section{The remaining case}\label{sec5}
Summarizing the results so far, Theorem~\ref{MT} holds true.
In the following, we focus on the variety as in Theorem~\ref{MT}\,(iv). We work under the following setting.

\begin{Set}\label{set:4a} Let $X \subset \P^N$ be an $11$-dimensional Fano variety $X \subset \P^{17}$ whose Fano index is $8$. Moreover, $X$ satisfies the following:
\begin{enumerate}\itemsep=0pt
\item[$({\rm i})$] $X$ is covered by lines.
\item[$({\rm ii})$] For a general point $x \in X$, we denote by $\cL_x$ the Hilbert scheme of lines on $X$ passing through $x \in X$. Then $\cL_x \subset \P\bigl((T_xX)^{\vee}\bigr)=\P^{10}$ is projectively equivalent to a Gushel--Mukai $6$-fold.
\end{enumerate}
The ample generator of the Picard group ${\rm Pic}(\cL_x)$ is denoted by $\cO_{\cL_x}(1)$. The embedding $\cL_x \subset \P((T_xX)^{\vee})=\P^{10}$ is given by the complete linear system $|\cO_{\cL_x}(1)|$.
Let $\pi_x\colon \cU_x \to \cL_x$ be the universal family and ${\rm ev}_x\colon \cU_x\to X$ the evaluation morphism:
\[
 \xymatrix{
 \cU_x \ar[r]^{{\rm ev}_x} \ar[d]_{\pi_x} & X \\
 \,\cL_x. &
} \]
\end{Set}
We denote by ${\rm Locus}(\cL_x)$ the image of ${\rm ev}_x$: ${\rm Locus}(\cL_x):={\rm ev}_x(\cU_x)$.

\begin{Proposition}\label{prop:GM} Under the setting of Setup~$\ref{set:4a}$, $X$ satisfies the following:
\begin{enumerate}\itemsep=0pt
\item[$(i)$] $H^2(X, \Z)\cong \Z[H]$ and $H^4(X, \Z)\cong \Z\bigl[H^2\bigr]$.
\item[$(ii)$] $c_1(X)=8H \in H^2(X, \Z)$ and $c_2(X)=31H^2\in H^4(X, \Z)$. In particular, the second Chern character ${\rm ch}_2(X)=\frac{1}{2}\bigl(c_1(X)^2-2c_2(X)\bigr)=H^2$, that is, $X$ is a $2$-Fano variety in the sense of {\rm \cite{ArCas12}}.
\item[$(iii)$] The self-intersection number $H^{11}$ is equal to $24$.
\item[$(iv)$] Through two general points of $X$, there passes an irreducible conic contained in $X$, that is, $X$ is a conic-connected variety in the sense of~{\rm \cite{IR10}}.
\end{enumerate}
\end{Proposition}

\begin{proof} (i) follows from the Barth--Larsen theorem \cite[Corollary]{Lar73}. Since the Fano index of $X$ is~$8$, we have $c_1(X)=8H \in H^2(X, \Z)$. Applying \cite[Proposition~1.3\,(1.2)]{ArCas12} or \cite[Proposition~4.2]{Dr06}, we have
\[
{\pi_x}_{\ast}{\rm ev}_x^{\ast}({\rm ch}_2(X))={c_1(\cO_{\cL_x}(1))}.
\]
This yields that ${\rm ch}_2(X)={H^2}$; thus we obtain $c_2(X)=31H^2\in H^4(X, \Z)$. As a consequence, we see that (ii) holds.

Let $Y\subset \P^{10}$ be a general $4$-dimensional linear section of $X \subset \P^{17}$ and $H_Y$ a divisor on $Y$ which is the restriction of $H$ to $Y$; then one see that $Y$ is a Fano $4$-fold such that $c_1(Y)=H_Y$ and $c_2(Y)=3H_Y^2$. By the Riemann--Roch formula and various vanishing theorems, we obtain the following equation (see \cite[p.~48]{Kuc97}):
\[
h^0(-K_Y)=1+\frac{(-K_Y)^2c_2(Y)}{12}+\frac{(-K_Y)^4}{6}=1+\frac{5}{12}H^{11}.
\]
On the other hand, according to \cite[Corollary~2]{BLL91}, the quadratic variety $X$ is arithmetically Cohen--Macaulay (aCM) (also called projectively Cohen--Macaulay), that is, the homogeneous coordinate ring of $X \subset \P^{17}$ is Cohen--Macaulay (see also the sentence before \cite[Theorem~3.8]{IR}). By \cite[Theorem~1.3.3]{Mig}, $Y\subset \P^{10}$ is also aCM and in particular linearly normal. Thus we obtain
\[
11=h^0(-K_Y)=1+\frac{5}{12}H^{11}.
\]
This yields that $H^{11}=24$. Thus, (iii) holds.

Finally, (iv) follows from \cite[Theorem~3.14]{HK05}.
\end{proof}

\begin{Proposition}
Let $X$ be a variety as in Setup~$\ref{set:4a}$. For a general point $x\in X\subset \P^{17}$, ${\rm Locus}(\cL_x)\subset {\mathbb T}_xX\cong \P^{11}$ is projectively equivalent to the cone ${\rm Cone}(o', \cL_x) \subset \P^{11}$, where $o'$ is a~point in $\P^{11}\setminus \P((T_xX)^{\vee})$.
\end{Proposition}

\begin{proof} According to \cite[Section~2.2]{ArCas12}, there exists a rank $2$ vector bundle $\cE$ over $\cL_x$ satisfying the following:
\begin{itemize}\itemsep=0pt
\item $\cU_x=\P(\cE)$.
\item The sequence $0\to \cO_{\cL_x} \to \cE\to \cO_{\cL_x}(-1) \to 0$ is exact.
\end{itemize}
Since we have
\[
\dim {\rm Ext^1}(\cO_{\cL_x}(-1), \cO_{\cL_x})=h^1(\cO_{\cL_x}(1))=0,
\]
$\cE$ is isomorphic to $\cO_{\cL_x}\oplus \cO_{\cL_x}(-1)$. By tensoring $\cO_{\cL_x}(1)$ to $\cE$, we see that $\cU_x$ is isomorphic to~$\P(\cO_{\cL_x}(1)\oplus \cO_{\cL_x})$. Denoting by $\cO_{\cU_x}(1)$ the tautological line bundle of $\cU_x=\P(\cO_{\cL_x}(1)\oplus \cO_{\cL_x})$, the evaluation morphism
\[
{\rm ev}_x\colon \ \cU_x=\P(\cO_{\cL_x}(1)\oplus \cO_{\cL_x}) \to {\rm Locus}(\cL_x)
\]
is determined by a linear system $\Lambda \subset |\cO_{\cU_x}(1)|$.
The complete linear system $|\cO_{\cU_x}(1)|$ determines the morphism $\cU_x=\P(\cO_{\cL_x}(1)\oplus \cO_{\cL_x})\to {\rm Cone}(o', \cL_x)$, which is nothing but the blow-up of ${\rm Cone}(o', \cL_x)$ at the vertex $o'$. Hence, the evaluation morphism ${\rm ev}_x$ factors through the morphism $\cU_x=\P(\cO_{\cL_x}(1)\oplus \cO_{\cL_x})\to {\rm Cone}(o', \cL_x)$, and the morphism $p\colon {\rm Cone}(o', \cL_x) \to {\rm Locus}(\cL_x)$ is given by a projection:
\[
 \xymatrix{
 \cU_x=\P(\cO_{\cL_x}(1)\oplus \cO_{\cL_x}) \ar[r] \ar[dr]_{{\rm ev}_x} & {\rm Cone}(o', \cL_x) \ar[d]^p \\
 & {\rm Locus}(\cL_x).
 }
\]

Since there exists one-to-one correspondence between the set of lines in ${\rm Cone}(o', \cL_x)$ passing through the vertex $o'$ and the set of lines in ${\rm Locus}(\cL_x)$ passing through $x$ via $p$, the map~$p$ is bijective. By the Zak's linear normality \cite[Chapter~II, Corollary 2.11]{Z} (see also \cite[Theorem~5.1.6]{RussoBk}), the secant variety ${\rm Sec}(\cL_x)$ coincides with the whole space $\P^{10}$. This implies that ${\rm Sec}({\rm Cone}(o', \cL_x))=\P^{11}$. Thus we conclude that $\Lambda=|\cO_{\cU_x}(1)|$ and ${\rm Locus}(\cL_x)\cong {\rm Cone}(o', \cL_x)$. As a consequence, our assertion holds.
\end{proof}

\begin{Lemma}\label{lem:GM} Let $X$ be a variety as in Setup~$\ref{set:4a}$. Then a double cover does not exist from $X$ to a smooth quadratic variety $Y\subset \P^{16}$.
\end{Lemma}

\begin{proof} Assume the contrary, that is, there exists a double cover $f\colon X\to Y$ to a smooth quadratic variety $Y\subset \P^{16}$. By Theorem~\ref{them:HC}, $Y$ should be a complete intersection; thus we have $\deg Y=32$. This yields that $\deg X=2 \deg Y=64$. This contradicts Proposition~\ref{prop:GM}~(iii).
\end{proof}

\begin{Remark}
Let $X$ be a variety as in Setup~$\ref{set:4a}$. By Lemma~\ref{lem:GM}, if we can show that $X$ has a~double cover as in Lemma~\ref{lem:GM}, we can see that case (iv) of Theorem~\ref{MT} does not occur.
\end{Remark}

\subsection*{Acknowledgments} The author would like to express his sincere gratitude to the anonymous referees for their meticulous reading of his manuscript. Their insightful comments and suggestions have significantly improved various parts of this work.
The author knew the proof of Proposition~\ref{prop:spinor} in \cite{Stack} and would like to thank them for teaching how to prove it. Additionally, the author would like to thank Professor Wahei Hara for explaining the parts of the proof that were not understood and for providing a detailed explanation. The author is partially supported by JSPS KAKENHI Grant Number 21K03170, 25K06940.

\pdfbookmark[1]{References}{ref}
\LastPageEnding


\begin{thebibliography}{99}
\footnotesize\itemsep=0pt

\bibitem{ArCas12}
Araujo C., Castravet A.M., Polarized minimal families of rational curves and
 higher {F}ano manifolds,
 \href{https://doi.org/10.1353/ajm.2012.0008}{\textit{Amer.~J. Math.}}
 \textbf{134} (2012), 87--107,
 \href{http://arxiv.org/abs/0906.5388}{arXiv:0906.5388}.

\bibitem{Bel}
Bertini E., Introduzione alla geometria proiettiva degli iperspazi con
 appendice sulle curve algebriche e loro singolarit\'a, Enrico Spoerri, Pisa,
 1907.

\bibitem{BLL91}
Bertram A., Ein L., Lazarsfeld R., Vanishing theorems, a theorem of {S}everi,
 and the equations defining projective varieties,
 \href{https://doi.org/10.2307/2939270}{\textit{J.~Amer. Math. Soc.}}
 \textbf{4} (1991), 587--602.

\bibitem{CK}
Choi Y., Kwak S., Remarks on the defining equations of smooth threefolds in
 {$\mathbb{P}^5$},
 \href{https://doi.org/10.1023/A:1022133506844}{\textit{Geom. Dedicata}}
 \textbf{96} (2003), 151--159.

\bibitem{Dr06}
Druel S., Classes de {C}hern des vari\'et\'es unir\'egl\'ees,
 \href{https://doi.org/10.1007/s00208-006-0772-5}{\textit{Math. Ann.}}
 \textbf{335} (2006), 917--935,
 \href{http://arxiv.org/abs/math.AG/0502207}{arXiv:math.AG/0502207}.

\bibitem{ESB}
Ein L., Shepherd-Barron N., Some special {C}remona transformations,
 \href{https://doi.org/10.2307/2374881}{\textit{Amer.~J. Math.}} \textbf{111}
 (1989), 783--800.

\bibitem{EH87}
Eisenbud D., Harris J., On varieties of minimal degree (a centennial account),
 in Algebraic Geometry, {B}owdoin, 1985 ({B}runswick, {M}aine, 1985),
 \textit{Proc. Sympos. Pure Math.}, Vol.~46,
 \href{https://doi.org/10.1090/pspum/046.1/927946}{American Mathematical
 Society}, Providence, RI, 1987, 3--13.

\bibitem{Fal81}
Faltings G., Ein {K}riterium f\"ur vollst\"{a}ndige {D}urchschnitte,
 \href{https://doi.org/10.1007/BF01394251}{\textit{Invent. Math.}} \textbf{62}
 (1981), 393--401.

\bibitem{Fuj1}
Fujita T., On the structure of polarized manifolds with total deficiency
 one.~{I}, \href{https://doi.org/10.2969/jmsj/03240709}{\textit{J.~Math. Soc.
 Japan}} \textbf{32} (1980), 709--725.

\bibitem{Fuj2}
Fujita T., On the structure of polarized manifolds with total deficiency
 one.~{II}, \href{https://doi.org/10.2969/jmsj/03330415}{\textit{J.~Math. Soc.
 Japan}} \textbf{33} (1981), 415--434.

\bibitem{Hart-ci}
Hartshorne R., Varieties of small codimension in projective space,
 \href{https://doi.org/10.1090/S0002-9904-1974-13612-8}{\textit{Bull. Amer.
 Math. Soc.}} \textbf{80} (1974), 1017--1032.

\bibitem{Har}
Hartshorne R., Algebraic geometry, \textit{Grad. Texts in Math.}, Vol.~52,
 \href{https://doi.org/10.1007/978-1-4757-3849-0}{Springer}, New York, 1977.

\bibitem{Hw}
Hwang J.-M., Geometry of minimal rational curves on {F}ano manifolds, in School
 on {V}anishing {T}heorems and {E}ffective {R}esults in {A}lgebraic {G}eometry
 ({T}rieste, 2000), \textit{ICTP Lect. Notes}, Vol.~6, Abdus Salam Int. Cent.
 Theoret. Phys., Trieste, 2001, 335--393.

\bibitem{HK05}
Hwang J.-M., Kebekus S., Geometry of chains of minimal rational curves,
 \href{https://doi.org/10.1515/crll.2005.2005.584.173}{\textit{J.~Reine Angew.
 Math.}} \textbf{584} (2005), 173--194,
 \href{http://arxiv.org/abs/math.AG/0403352}{arXiv:math.AG/0403352}.

\bibitem{IR10}
Ionescu P., Russo F., Conic-connected manifolds,
 \href{https://doi.org/10.1515/CRELLE.2010.054}{\textit{J.~Reine Angew.
 Math.}} \textbf{644} (2010), 145--157,
 \href{http://arxiv.org/abs/math.AG/0701885}{arXiv:math.AG/0701885}.

\bibitem{IR}
Ionescu P., Russo F., Manifolds covered by lines and the {H}artshorne
 conjecture for quadratic manifolds,
 \href{https://doi.org/10.1353/ajm.2013.0020}{\textit{Amer.~J. Math.}}
 \textbf{135} (2013), 349--360,
 \href{http://arxiv.org/abs/1209.2047}{arXiv:1209.2047}.

\bibitem{KO}
Kobayashi S., Ochiai T., Characterizations of complex projective spaces and
 hyperquadrics, \href{https://doi.org/10.1215/kjm/1250523432}{\textit{J.~Math.
 Kyoto Univ.}} \textbf{13} (1973), 31--47.

\bibitem{KM}
Koll\'ar J., Mori S., Birational geometry of algebraic varieties,
 \textit{Cambridge Tracts in Math.}, Vol.~134,
 \href{https://doi.org/10.1017/CBO9780511662560}{Cambridge University Press},
 Cambridge, 1998.

\bibitem{Kuc97}
K\"uchle O., Some remarks and problems concerning the geography of {F}ano
 {$4$}-folds of index and {P}icard number one, \textit{Quaestiones Math.}
 \textbf{20} (1997), 45--60.

\bibitem{Kuz18}
Kuznetsov A., On linear sections of the spinor tenfold.~{I},
 \href{https://doi.org/10.4213/im8756}{\textit{Izv. Math.}} \textbf{82}
 (2018), 694--751.

\bibitem{KP18}
Kuznetsov A., Perry A., Derived categories of {G}ushel--{M}ukai varieties,
 \href{https://doi.org/10.1112/s0010437x18007091}{\textit{Compos. Math.}}
 \textbf{154} (2018), 1362--1406,
 \href{http://arxiv.org/abs/1605.06568}{arXiv:1605.06568}.

\bibitem{Lands96}
Landsberg J.M., Differential-geometric characterizations of complete
 intersections,
 \href{https://doi.org/10.4310/jdg/1214458739}{\textit{J.~Differential Geom.}}
 \textbf{44} (1996), 32--73.

\bibitem{Lar73}
Larsen M.E., On the topology of complex projective manifolds,
 \href{https://doi.org/10.1007/BF01390209}{\textit{Invent. Math.}} \textbf{19}
 (1973), 251--260.

\bibitem{Lich82}
Lichtenstein W., A system of quadrics describing the orbit of the highest
 weight vector, \href{https://doi.org/10.2307/2044044}{\textit{Proc. Amer.
 Math. Soc.}} \textbf{84} (1982), 605--608.

\bibitem{Stack}
Mathematics Stack Exchange, Quadratic equations defining a smooth hyperplane
 section of the 10-dimensional spinor variety, 2024, available at
 \url{https://math.stackexchange.com/questions/4899260/quadratic-equations-defining-a-smooth-hyperplane-section-of-the-10-dimensional-s}.

\bibitem{Mig}
Migliore J.C., An introduction to deficiency modules and liaison theory for
 subschemes of projective space, \textit{Lect. Notes Ser.}, Vol.~24, Seoul
 National University, Seoul, 1994.

\bibitem{Mu}
Mukai S., Biregular classification of {F}ano {$3$}-folds and {F}ano manifolds
 of coindex~{$3$}, \href{https://doi.org/10.1073/pnas.86.9.3000}{\textit{Proc.
 Nat. Acad. Sci. USA}} \textbf{86} (1989), 3000--3002.

\bibitem{Mum:quad}
Mumford D., Varieties defined by quadratic equations, in Questions on
 {A}lgebraic {V}arieties ({C}.{I}.{M}.{E}., {III} {C}iclo, {V}arenna, 1969),
 Centro Internazionale Matematico Estivo (C.I.M.E.), Edizioni Cremonese, Rome,
 1970, 29--100.

\bibitem{Net}
Netsvetaev N.Yu., Projective varieties defined by small number of equations are
 complete intersections, in Topology and Geometry~-- {R}ohlin {S}eminar,
 \textit{Lecture Notes in Math.}, Vol.~1346,
 \href{https://doi.org/10.1007/BFb0082787}{Springer}, Berlin, 1988, 433--453.

\bibitem{RussoBk}
Russo F., On the geometry of some special projective varieties, \textit{Lect.
 Notes Unione Mat. Ital.}, Vol.~18,
 \href{https://doi.org/10.1007/978-3-319-26765-4}{Springer}, Cham, 2016.

\bibitem{Z}
Zak F.L., Tangents and secants of algebraic varieties, \textit{Transl. Math.
 Monogr.}, Vol.~127, \href{https://doi.org/10.1090/mmono/127}{American
 Mathematical Society}, Providence, RI, 1993.

\end{thebibliography}
\end{document}